\documentclass[11 pt]{amsart}
\usepackage{amsfonts}
\usepackage{latexsym,amsmath,amsfonts}
\usepackage{amssymb}
\usepackage{graphicx}
\usepackage[dvips]{color}
\newtheorem{theorem}{Theorem}[section]
\newtheorem{thm}[theorem]{Theorem}
\newtheorem{pro}[theorem]{Proposition}

\newtheorem{lemma}[theorem]{Lemma}

\newtheorem{defi}[theorem]{Definition}

\newtheorem{example}[theorem]{Example}
\numberwithin{equation}{section}

\def\m{\medskip}

\def\cal{\mathcal }
\def\R{\mathbb R}
\def\N{\mathbb N}

\def\Q{\mathbb Q}

\def\mathscr{\mathcal }
\def\CJ{\mathcal J}

\def\m{\mathbf m}
\def\1{\mathbf 1}
\def\bj{\mathbf j}

\newcommand{\blambda}{{\boldsymbol{\lambda}}}

\newcommand{\balpha}{{\boldsymbol{\alpha}}}
\newcommand{\brho}{{\boldsymbol{\rho}}}
\newcommand{\btau}{{\boldsymbol{\tau}}}

\newcommand{\bm}{{\mathbf{m}}}
\newcommand{\bi}{{\mathbf{i}}}

\newcommand{\D}{{\mathcal D}}

\begin{document}

\title[Frobenius problem and Lipschitz equivalence]{Higher dimensional Frobenius problem and Lipschitz equivalence of Cantor sets}


\author{Hui Rao} \address{Department of Mathematics, Hua Zhong Normal University, Wuhan 430072,
China} \email{ hrao@@mail.ccnu.edu.cn}

\author{Yuan Zhang$\dagger$} \address{Department of Mathematics, Hua Zhong Normal University, Wuhan 430072,
China} \email{ }
\date{\today}
\thanks {The work is supported by CNSF No. 11171128 and NSFC No. 11431007.}


\thanks{$\dagger$ The correspondence author.}

\begin{abstract}  The higher dimensional Frobenius problem was introduced
 by a preceding paper [Fan, Rao and Zhang, Higher dimensional Frobenius problem:
maximal saturated cones, growth function and rigidity, Preprint 2014].

In this paper, we investigate the Lipschitz equivalence of dust-like self-similar sets in $\R^d$.
For any self-similar set, we associate with it a higher dimensional Frobenius problem, 
and we show that the directional growth function of the associate higher dimensional Frobenius problem is a Lipschitz invariant.

As an application, we solve the Lipschitz equivalence problem when two dust-like self-similar sets $E$ and $F$ have
  coplanar ratios, by showing that they are Lipschitz equivalent if and only if  the contraction vector of  the $p$-th iteration
 of  $E$ is a permutation of that
 of the $q$-th iteration of  $F$ for some $p, q\geq 1$. This partially answers a question raised by Falconer and Marsh 
  [On the Lipschitz equivalence of
Cantor sets, \emph{Mathematika,} \textbf{39} (1992), 223--233].
\end{abstract}

\maketitle
\section{introduction}
Let $ E,F$ be subsets of $\R^d $. We say that $ E$ and $F $ are \emph{Lipschitz equivalent},
write $E\sim F$,
 if there exists a bijection $ f: E\rightarrow F$
and a constant $ C>0$ such that
\begin{equation}
 C^{-1}\mid x-y \mid \leq  \mid f(x)-f(y)\mid \leq  C\mid x-y \mid
\end{equation}
for all $ x,y \in E.$

An area of interest in the study of self-similar sets is the
Lipschitz equivalence property. With Lipschitz equivalence many
important properties of a self-similar set are preserved.
The earliest works on this area are  Cooper and Pignataro \cite{CP} (1988),
Falconer and Marsh \cite{FaMa92} (1992),   and
  David and Semmes \cite{DS} (1997). A survey on recent progress can be found in \cite{RRW13}.

There are two
different types of problems in this area.
 The first type problem, raised by Falconer and Marsh \cite{FaMa92},
 assumes that  two self-similar sets $E$ and $F$ have nice topological property, precisely,
they are dust-like, and asks how the Lipschitz equivalence relates to the contraction ratios of $E$ and $F$.
The second type problem, initialled by David and Semmes \cite{DS}, assumes that $E$ and $F$ have the same contraction
ratios, and asks how the geometrical placements of the branches affect the Lipschitz equivalence.

The first progress on the second type problem was made by  Rao, Ruan and Xi \cite{RRX06}, which  solved
the so-called $\{1,3,5\}-\{1,4,5\}$ problem posed in \cite{DS}. After that, there are many
generalizations and further progresses, for example,  \cite{XiRu07, RWX12, DengGT1} on one dimensional case, \cite{XiXi10, Ro10,XiXi12}
on higher dimensional case,  \cite{DengGT2, Lau13} on self-similar sets that are not totally discrete.

The first type problem is more tricky, and there is no progress until recent works of
Rao, Ruan and Wang \cite{RRW12} (2012) and Xiong and  Xi \cite{XiXi13} (2013).
The work \cite{RRW12} introduced a matchable condition, and showed that
  two dust-like self-similar sets must satisfy a  matchable condition  if
they are Lipschitz equivalent; as applications, the authors solved the problem if both self-similar sets have full rank or both of them
are two-branch self-similar sets.  Xiong and  Xi \cite{XiXi13}  studied the problem
when  $E$ and $F$ have rank $1$. They showed that if the Hausdorff dimension is fixed,
the number of different Lipschitz equivalence class equals the class number of the field generated by the ratios.

For related works
on Lipschitz equivalence of other fractals, see \cite{FaMa89} on quasi-circles, \cite{Xi04} on self-conformal sets,
\cite{MS08, Llo09} on bi-Lipschitz embedding of self-similar sets,
\cite{RRY08} on general Cantor sets, \cite{LLM} on Bedford-McMullen carpets.


Recall that  a self-similar set is the attractor of an
iterated  function system (IFS). Let $\{\phi_j\}_{j=1}^{m}$
 be an IFS on $\mathbb{R}^d ~$,  where each $ \phi_j$ is a contractive similarity with
contraction ratio $ 0< \rho_j<1$.
The attractor of the IFS is the unique nonempty compact set $E$ satisfying
$E=\bigcup_{j=1}^m \phi_j(E)$.
We say that the attractor $E$ is \emph{dust-like}, or alternatively,
the IFS $\{\phi_j\}$ satisfies the \emph{strong separation condition},
if the sets $ \{\phi_j( E)\}$ are disjoint. It is well known that if
$E$ is dust-like, then the Hausdorff dimension $ \delta=\dim_{H}(E)$ of $E$ satisfies
$\sum_{j=1}^m \rho_j^\delta=1.$ See \cite{Hut81, Fal-book2}.

For any $ \rho_1,\dots,\rho_m \in (0,1)$ with $ \sum_{j=1}^m \rho_j^d<1$,
we will call $ \boldsymbol{\rho}=(\rho_{1},\cdot\cdot\cdot,\rho_{m})$ a \emph{contraction vector},
and use the notation ${\cal D}(\boldsymbol{\rho})={\cal D}( \rho_1, \dots,\rho_m) $
to denote the set of all dust-like self-similar sets with contraction ratios
$ \rho_j, \ j=1,\dots,m $. All sets in $ {\cal D}(\boldsymbol{\rho})$
have the same Hausdorff dimension, which we denote by $\dim_{H}{\cal D}(\boldsymbol{\rho})$.
It is well-known that the elements in ${\cal D}(\rho_1,\dots,\rho_m)$   are Lipschitz equivalent to each other;
hence we denote  ${\cal D}(\boldsymbol{\rho})\sim {\cal D}(\boldsymbol{\tau})$,
if $E \sim F$ for some (and thus for all) $E \in {\cal D}(\boldsymbol{\rho})$ and $F \in {\cal D}(\boldsymbol{\tau}).$
See \cite{FaMa92, RRW12}.

In the  study of Lipschitz equivalence, a main idea is to construct Lipschitz invariants.
Falconer and Marsh \cite{FaMa92} introduced  Lipschitz invariants related to the algebraic properties of the contraction ratios.
Let $\langle  \rho_1,\dots,\rho_m  \rangle$, or $\langle\brho\rangle$,  denote the subgroup of
 $( \mathbb{R}^{+} ,\times)~$ generated by $ \rho_1,\dots,\rho_m$;
  let $sgp(a_1,\dots, a_m)$ denote the multiplicative semi-group generated by $a_1,\dots, a_m$;
let $\mathbb{Q}(a_1,\dots,a_m)$ denote the subfield of $\mathbb{R}$ generated by $\mathbb{Q}$
and $a_1,\dots,a_m$.

The \emph{rank} of $\langle\brho\rangle$,
 which we denote by rank $\langle \boldsymbol{\rho}\rangle$, is defined  to be the cardinality of
 the basis of $\langle  \rho_{1},\cdot\cdot\cdot,\rho_{m}  \rangle~$.

\begin{pro}\label{FM-theo} (\cite{FaMa92})
Let ${\cal D}(\rho_1,\dots, \rho_m )\sim {\cal D}(\tau_1,\dots, \tau_n) $,
and  $\delta$ be their common Hausdorff dimension.  Then
\begin{enumerate}

\item[$(i)$] There exist positive integers $p,~q$ such that
$$
sgp(\rho_1^p,\dots,\rho_m^p)\subseteq sgp(\tau_1,\dots,\tau_n),\quad
 sgp(\tau_1^q,\dots,\tau_n^q)\subseteq sgp(\rho_1,\dots,\rho_m);
$$

 \item[$(ii)$] $\mathbb{ Q}(\rho_1^\delta,\dots,\rho_m^\delta)=\mathbb{ Q}(\tau_1^\delta,\dots,\tau_n^\delta).$
 \end{enumerate}
\end{pro}

Property (ii)  shows that $\Q(\rho_1^\delta,\dots, \rho_m^\delta)$ is a Lipschitz invariant.
In Section 2, we show that item (i) can also be made into a Lipschitz invariant. Define
$$ V_{\brho}^+=\Q^+ \log \rho_1+\cdots+\Q^+ \log \rho_m.$$

\begin{theorem}\label{thm-V+}  Let $\brho$ and $\btau$ be two contraction vectors.
Property (i) in Theorem \ref{FM-theo} holds if and only if
$V_{\brho}^+=V_{\btau}^+.$
\end{theorem}

Xi and
Ruan \cite{XiRu08} observed that  bi-Lipschitz maps between two dust-like self-similar
sets $E$ and $F$ enjoy a certain measure-preserving property.
Using this property,
 Rao, Ruan and Wang \cite{RRW12}
 constructed a family of relations between symbolic spaces related to two dust-like  self-similar sets $E$ and $F$,
and showed that these relations must satisfy a \emph{matchable condition}  if
they are Lipschitz equivalent. As applications, they show that

\begin{pro}\label{pro-RRW} (\cite{RRW12})
Let $(\rho_1,\dots,\rho_m )$ and $(\tau_1,\dots,\tau_n)$ be two contraction vectors.

(i)   If $\text{rank}\langle\brho\rangle=m$ and $\text{rank} \langle \btau\rangle=n$, \textit{i.e.}, both vectors have full rank, then   ${\mathcal D}(\brho)\sim {\mathcal D}(\btau)$  if and only if $\btau$ is a permutation of $\brho$.

(ii) If $m=n=2$, then    ${\mathcal D}(\brho)\sim {\mathcal D}(\btau)$ if and only if
either $\btau$ is a permutation of $\brho$, or there exists a real number $0<\lambda<1$ such that
    $
       \{\rho_1,\rho_2\}=\{\lambda^5,\lambda\},\quad  \{\tau_1,\tau_2\}= \{ \lambda^3,\lambda^2\}.
    $
\end{pro}


As the existing results show, the Lipschitz equivalence of self-similar sets
are tightly related to the multiplicative groups generated by contraction ratios. Instead of working with the multiplicative subgroup $\langle \brho\rangle$
it is more practical for us to work with the additive groups of
$(\mathbb{Z}^s,+),$ which is associated to $\langle \brho\rangle$ ($s$ being the rank of $\langle \brho\rangle$).
For this purpose, \cite{RRW12} introduced  the notion of pseudo-basis.

We call positive numbers $\omega_1,\dots,\omega_s \in \R$ a \emph{pseudo-basis} of $\langle \boldsymbol{\rho}\rangle$, or of $\brho$ in short,
 if
 $\langle \boldsymbol{\rho}\rangle \subset \langle  \omega_1,\dots,\omega_{s} \rangle$
 and $\text{rank~} \langle \boldsymbol{\rho}\rangle=s$.
Clearly, $V_{\brho}^+=V_{\btau}^+$ implies that $\brho$ and $\btau$
  have a \emph{common} pseudo-basis.

Actually, we shall fix a (common) pseudo-basis $\blambda=(\lambda_1,\dots, \lambda_s)$, then
the group $\langle \blambda\rangle$ is isomorphic to $(\mathbb{Z}^s,+),$
and $\langle \brho\rangle$ is isomorphic to a subgroup of  $(\mathbb{Z}^s,+).$
This leads to the following notations.

For $x=(x_1,\dots,x_s)\in \mathbb{Z}^s,$
we define $\text{exp}_{\blambda}: \mathbb{Z}^s\rightarrow \langle \blambda\rangle$ by
\begin{equation}
\boldsymbol{\lambda}^x=\prod_{i=1}^s{\lambda_i}^{x_i}.
\end{equation}
Also, we define the inverse function $\log_{\boldsymbol{\lambda}}: \langle \blambda\rangle\rightarrow \mathbb{Z}^s$ as
\begin{equation}
 \log_{\boldsymbol{\lambda}}z=x, \text{ where } \ z=\boldsymbol{\lambda}^x \in \langle \blambda\rangle.
\end{equation}

For $X_1,\dots, X_m\in \R^s$,  we shall use the notation $X=\{X_1,\dots, X_m\}$ and denote
$$
{\mathbf C}_X= \R^+X_1+\cdots+\R^+X_m.
$$
 to be the cone generated by $X_1,\dots, X_m$, where $\R^+$ denotes the set of non-negative real numbers. We note that $X_j$'s are not required to be distinct.

In the preceding  paper  Fan, Rao and Zhang \cite{FRZ14}, we introduced the higher dimensional Frobenius problem and investigated its various properties.
Especially, a \emph{directional growth function} $\gamma_{_X}$ is defined (see Section 3).
For a contraction vector $(\rho_1,\dots, \rho_m)$ with  a pseudo-basis $(\lambda_1,\dots, \lambda_s)$, we can associate with it a higher dimensional Frobenius problem as follows. Set
\begin{equation}
X_i=\log_{\boldsymbol{\lambda}}\rho_i, \quad i=1, \dots, m.
\end{equation}
Put $\balpha=-(\log \lambda_1,\dots,\log \lambda_s ),$ then $X_j\cdot \balpha =-\log \rho_j>0$ for all $j$,
where $\cdot$ denotes the inner product in $\R^s$.
 Therefore,  the vectors $X_i$    are located in an (open)  half-space of $\R^s$ and
 a higher dimensional Frobenius problem can be defined by the defining data
 $X=\{X_1,\dots, X_m\}$,   called
the \emph{associate higher dimensional Frobenius problem}.


Using the matchable condition in \cite{RRW12}, we show that  the directional growth function $\gamma$ is a Lipschitz invariant,
which is the main result of this paper.

\begin{thm} \label{main-1}  Let  $(\rho_1,\dots,\rho_m )$ and $(\tau_1,\dots,\tau_n)$ be two contraction vectors
 such that ${\mathcal D}(\brho)\sim {\mathcal D}(\btau)$,  let $\blambda$ be a  common pseudo-basis of
$\brho$ and  $\btau$. Denote
$$
X_j=\log_\blambda \rho_j(1\leq j\leq m), \quad Y_k=\log_\blambda \tau_k(1\leq k\leq n).
$$
Then

(i)  ${\mathbf C}_X={\mathbf C}_Y$, where $X=\{X_1,\dots, X_m\}$ and $Y=\{Y_1,\dots, Y_n\}$;

(ii) $ \gamma_{_X}(\theta)= \gamma_{_Y}(\theta)$
 for all unit vector $\theta\in {\mathbf C}_X.$
\end{thm}

Comparing to the matchable condition, the function  $\gamma$ is much easier to handle.
Especially, explicit formulas of $\gamma$ are obtained in the so-called coplanar case in \cite{FRZ14}.

\begin{defi}{\rm
We say a contraction ratio $\brho$ (of rank $s$) is \emph{coplanar}, if there exists a pseudo-basis $\blambda$ of $\langle \brho \rangle$,
such that $\{\log_{\blambda} \rho_j\}$ locate in a common hyperplane of ${\mathbb R}^s$.}
\end{defi}

The coplanar property is independent of the choice of the pseudo-basis,
for if $\blambda'$ is another pseudo-basis of $\langle \brho \rangle$, then there is an invertible matrix $L$
such that $\log_{\blambda'} \rho_j=L \log_{\blambda} \rho_j$ for all $j$.

 We define the \emph{$k$-th iteration} of $\brho$, denoted by
$\brho^k$, to be the vector
$$
(\rho_{\bi})_{\bi\in\{1,\dots,m\}^k},
$$
where $\rho_{i_1\dots i_k}=\prod_{j=1}^k\rho_{i_j}$, and $\bi$ is ordered by the lexicographical order. Fan, Rao and Zhang \cite{FRZ14} proved that,
 in the coplanar case, the directional growth function completely determines the defining data, see Theorem 1.6 in \cite{FRZ14} (or
 Theorem \ref{rigidity} in Section 3).
As a consequence, we have

\begin{theorem}\label{main-2} If $\brho$ and $\btau$ are coplanar.  Then $\D(\brho)\sim \D(\btau)$  if and only if,
 there exist $p$ and $q$ such that the $p$-th iteration of $\brho$ is a permutation of the $q$-th  iteration of $\btau$.
\end{theorem}

As the following example shows, Proposition \ref{pro-RRW} (i), one of the main results in  \cite{RRW12},
 is a very special case of Theorem \ref{main-2}.

\begin{example}\label{exam-full} {\rm Let $a_j, k_j, b_j, \ell_j$ ($1\leq j\leq s$) be positive integers such that
$$
\brho=(\rho_1,\dots, \rho_m)=\left (\underbrace{\lambda_1^{k_1},\dots, \lambda_1^{k_1}}_{a_1},\dots,
\underbrace{\lambda_s^{k_s},\dots, \lambda_s^{k_s}}_{a_s} \right ),
$$
$$
\btau=(\tau_1,\dots, \tau_n)=\left (\underbrace{\lambda_1^{\ell_1},\dots, \lambda_1^{\ell_1}}_{b_1},\dots,
\underbrace{\lambda_s^{\ell_s},\dots, \lambda_s^{\ell_s}}_{b_s} \right ).
$$
Then ${\mathcal D}(\brho)\sim {\mathcal D}(\btau)$ if and only if $\brho$ is a permutation of $\btau$.
When all $a_j=b_j=1$, we obtain Proposition \ref{pro-RRW}(i).
(We leave the simple proof to Section \ref{sec-final}.)
}
\end{example}


The paper is organized as follows. In Section 2, we introduce vector spaces as  Lipschitz invariants; Theorem \ref{thm-V+} and Theorem \ref{main-1}(i) are proved there.
In Section 3, we recall the notations and results on higher dimensional Frobenius problem. Our main result, Theorem \ref{main-1}, is proved in Section 4. Theorem \ref{main-2} is proved in Section 5.

\section{\textbf{Lipschitz invariants}}

First,  we  show that vector spaces can serve  as   Lipschitz invariants. Let $(\rho_1,\dots, \rho_m)$ be a contraction vector, define
$$ V_{\brho}=\Q \log \rho_1+\cdots+\Q \log \rho_m.$$
Then $(V_{\brho},\mathbb{Q})$ is a vector space.

If  $\brho=(\rho_1,\dots,\rho_m)$ and $\btau=(\tau_1,\dots,\tau_n)$  possessing a common
pseudo-basis $\blambda=(\lambda_1,\dots, \lambda_s)$, then
$$ V_{\brho}=V_\btau=\Q \log \lambda_1+\cdots+\Q \log \lambda_s,$$
 and so   $(V_{\brho},\mathbb{Q})$ is a Lipschitz invariant.

Theorem \ref{thm-V+} asserts that $V_\brho^+=\Q^+ \log \rho_1+\cdots+\Q^+ \log \rho_m$ is also a Lipschitz invariant.

\medskip

\noindent \textbf{Proof of Theorem \ref{thm-V+}.}
Assume that Property (i) of Proposition \ref{FM-theo} holds.  Notice that $\tau_1,\dots,\tau_n$
belong to the semigroup generated by $\rho_1^{1/q},\dots,\rho_m^{1/q}.$ Let $a_{ji},1\leq j \leq n,1\leq i \leq m,$
be non-negative integers such that $\tau_j=\prod_{i=1}^m(\rho_i^{1/q})^{a_{ji}},$ then
$$
\log \tau_j=\sum_{i=1}^m\frac{a_{ji}}{q}\log \rho_i \in V^+_{\rho}.
$$
and so that $V^+_{\tau}\subset V^+_{\rho}.$ By symmetry,
we have $V^+_{\rho} \subset V^+_{\tau}.$

Next, we prove the other direction. Let $a_{ji}\in\mathbb{Q}^+,1\leq j \leq n,1\leq i \leq m,$ such that
$\log \tau_j=\sum_{i=1}^m a_{ji}\log \rho_i.$
Let $q>0$ be an integer such that $qa_{ji}\in \mathbb{Z}^+$ for all $i,j$.
Then $\tau_j^q\in sgp(\rho_1,\dots,\rho_m)$, hence $sgp(\tau_1^q,\dots,\tau_n^q)\subseteq sgp(\rho_1,\dots,\rho_m).$
By symmetry, there exists a positive integer $p$ such that $sgp(\rho_1^p,\dots,\rho_m^p)\subseteq sgp(\tau_1,\dots,\tau_n).$
$\Box$

\medskip

\medskip

\noindent \textbf{Proof of Theorem \ref{main-1} (i).}
By Proposition \ref{FM-theo} (i), there exist integers $p, k_1,\dots, k_n\geq 1$ such that
$\rho_1^p=\tau_1^{k_1}\cdots \tau_n^{k_n}$. Applying $\log_{\blambda}$ to both sides of the equation, we obtain
$$X_1=\frac{k_1}{p}Y_1+\cdots +\frac{k_n}{p}Y_n.
$$
Therefore $X_1\in {\mathbf C}_Y$. By the same reason $X_j\in {\mathbf C}_Y$ for all $j$
and so  ${\mathbf C}_X\subset {\mathbf C}_Y$. Finally, by symmetry, we obtain
${\mathbf C}_X= {\mathbf C}_Y$.
$\Box$

\bigskip

\section{\textbf{Higher dimensional Frobenius Problem}}
Let $a_1,\dots,a_m$ be positive integers, and assume they are coprime without loss of generality.
 Set
$$
{\mathcal J}=a_1\mathbb{N}+\dots+a_m\mathbb{N},
$$
where $\N=\{0,1,2,\dots\}$ denotes the set of natural numbers. Clearly ${\mathcal J}\subset \N,$ and
there exists a minimum integer $g=g(a_1,\dots,a_m)$ such that
$
(g+1+\mathbb{N})\subset {\mathcal J}.
$
Finding the value of  $g$ is the famous \textit{Frobenius problem}.
See for instance,   Ram\'irez-Alfons\'in \cite{Alf05}.

Fan, Rao and Zhang \cite{FRZ14} introduced the higher dimensional  Frobenius problem and investigated the basic properties.
 Let $X_1,\dots,X_m$   in $\mathbb{Z}^s$ be vectors locating in a half-space, that is, there is a non-zero vector
$\balpha\in {\mathbb R}^s$ such that
$\ X_j\cdot \balpha>0$  for all $j=1,\dots, m,$
and assume that $X_1,\dots, X_m$ span the space $\R^s$.
The concern is to
understand the
 structure of   the semi-group
\begin{equation}\label{equ9}
\CJ=X_1\mathbb{N}+\dots+X_m\mathbb{N}.
\end{equation}

%

To a higher dimensional Frobenius problem,  \cite{FRZ14} defines a  \emph{directional growth function},  which is useful for our purpose.
In the rest of this section, we review the definitions and results of \cite{FRZ14}.

\subsection{Multiplicity}
Let  $\Sigma_m^*:=\bigcup_{k=0}^{\infty}\{1,2,\dots,m \}^k$ be the set of words over the alphabet
$\{1,2,\dots,m \}$,
which can be also considered as a tree.
For any word $\mathbf{i}=i_1\dots i_n \in \Sigma_m^*,$ we define
\begin{equation}\label{kappa}
\kappa (\mathbf{i})=X_{i_1}+\cdots+X_{i_n}.
\end{equation}
We consider $\kappa: \Sigma_m^* \to \mathbb{Z}^s$ as the walk in $\mathbb{Z}^s$ guided by $X_1,\dots,X_m$ along with the tree $\Sigma_m^*$.
Elements in $\Sigma_m^*$ are also called {\em pathes} of the walk and $\kappa (\mathbf{i})$ is called the {\em visited position}
following the path $\mathbf{i} $.
Clearly, a point $z\in \mathbb{Z}^s$ is a visited position (of some path) if and only if
  $z\in {\mathcal J}$. We define
the \emph{multiplicity} of a point $z\in\mathcal J$ to be
\begin{equation}\label{equ10}
{\mathbf m}(z):=\# \{\omega\in \Sigma_m^*; \ \kappa(\omega)=z\},
\end{equation}
which measures how many times a position is visited.
For point $x\in {\mathbf C}_X$ but not in ${\mathcal J}$, instead of setting ${\mathbf m}(x)=0$,
  we define  ${\mathbf m}(x)$ to be the multiplicity of the point in $\CJ$ which
is nearest to $x$, that is,
\begin{equation}
{\mathbf m}(x):=\min \{{\mathbf m}(z); \ z\in {\mathcal J} \text{ and } |x-z|=d(x,{\mathcal J})\},
\end{equation}
where $|\cdot|$ denote the Euclidean norm and $d(x,{\mathcal J}):=\min\{|x-z|; ~z\in {\mathcal J}\}$.

The following theorem asserts that $\m(z)$ does not vary dramatically,  which plays an important role in \cite{FRZ14} as well as in the present paper.

\begin{thm}\label{Q(x)} (\cite{FRZ14})
Let $C_0$ be a positive integer.  There exists a polynomial $Q(x)$ with positive coefficients such that
$$
 \frac{1}{Q(|z|)}\leq\frac{{\mathbf m}(z)}{{\mathbf m}({z^{'}})}\leq Q(|z|)
$$
provided that $z,z'\in {\mathcal J}$ and $|z-z'| <C_{0}.$
\end{thm}

\subsection{Directional growth function}
The directional growth function  $\gamma(\theta)$ defined below
 describes the exponential increasing speed of the multiplicity along the direction $\theta$.

\begin{defi} For a unit vector $\theta\in {\mathbf C}_{X}$,
the  \emph{directional growth function} is defined as
\begin{equation}\label{gamma}
\gamma(\theta)=\underset{k \rightarrow  \infty}{\lim}\frac{\log \mathbf{m}(k\theta)}{k}
\end{equation}
as soon as the limit exists.
\end{defi}

 It is shown \cite{FRZ14} that the above limit always exists. Moreover,
according to Theorem \ref{Q(x)}, the limit in \eqref{gamma} still exists if $k$ tends to infinity  in $\R^+$ instead of in $\N$.

In general, it is difficult to obtain an explicit formula of $\gamma(\theta)$; nevertheless,
explicit formulas are obtained in a special case called coplanar case.

We say $X_1,\dots, X_m$ are \emph{coplanar}, if they locate at a same hyper-plane, \textit{i.e.},
there exists a vector $\eta \in \R^s$ such that
$$
\langle \eta,  X_j \rangle =1, \quad j=1,\dots, m.
$$
In this case, \cite{FRZ14} showed that
\begin{equation}\label{entropy}
\gamma(\theta) = {\langle \theta, \eta \rangle } \sup \left \{ h(p);~ p_1X_1+\cdots+p_mX_m=\frac{\theta}{\langle \theta, \eta\rangle} \right \},
\end{equation}
 where $p=(p_1,\dots, p_m)$ is a probability vector (we allow $p_j$ to take value $0$), and $h(p)$ is  the entropy of $p$ defined as
$
h(p)=-\sum_{j=1}^m p_j\log p_j.
$

%

\subsection{Rigidity results}
Given two collections of vectors
$X=\{X_1,\dots, X_m\}$ and
$Y=\{Y_1,\dots, Y_{n}\}$, if they define the same directional growth function, \textit{i.e.},
$$
{\mathbf C}_X={\mathbf C}_Y \text{ and } \gamma_X=\gamma_Y
$$
 what can we say about $X$ and $Y$?  Rao, Ruan and Wang \cite{RRW12} essentially obtained the following rigidity result.

 \begin{theorem} (\cite{RRW12})   Suppose    $X=\{X_1,\dots, X_s\}$ and
$Y=\{Y_1,\dots, Y_{s}\}$  are two collections of linearly independent vectors in ${\mathbb Z}^s$.
If they define the same directional growth function, then $X$ is a permutation of $Y$.
 \end{theorem}

The  rigidity property   also holds in the coplanar case (but the proof is much more difficult).
We define the $p$-th \emph{iteration} of $X$ to be the vector
 $$X^{(p)}=\left ( X_{i_1}+\cdots+X_{i_p}\right )_{{i_1\dots i_p}\in \{1,\dots,m\}^p}.$$
For example, the second iteration of $X=\left \{ (1,0),(0,1)\right \}$ is $\left \{(2,0),(1,1),(1,1),(0,2)\right \}$.

\begin{theorem}\label{rigidity} (\cite{FRZ14})
 Suppose  both $X=\{X_1,\dots, X_m\}$
and $Y=(Y_1,\dots, Y_{n})$ are coplanar, and   they
  define the same directional growth function. Then
there exist integers $p,q\geq 1$ such that
 the $p$-th iteration of $X$ is a permutation of the $q$-th iteration of $Y$.
\end{theorem}


\section{\textbf{Proof of Theorem \ref{main-1} (ii)}}

Let
$ \boldsymbol{\rho}=(\rho_1,\dots,\rho_m)$ be a contraction vector,
 $E\in{\cal D}(\boldsymbol{\rho})$,
and  $\boldsymbol{\lambda}=(\lambda_1,\dots,\lambda_s)$ be a pseudo-basis of $\langle \boldsymbol{\rho} \rangle$.
 Set
 $$
X_j=\log_{\boldsymbol{\lambda}}\rho_j ~(j=1,\dots,m)
 \text{ and }
 \balpha=-(\log \lambda_1,\dots, \log \lambda_s)
$$
as  in Section 1.
Then  a higher dimensional Frobenius problem can be defined.

For a word $ \mathbf{i}=i_1\dots i_k \in \Sigma_m^*,$
we define $ \rho_{\mathbf{i}}=\prod_{j=1}^k \rho_{i_j};$ then
\begin{equation}\label{eq-r2z}
\log_\blambda \rho_{\mathbf i}=\kappa ({\mathbf i})=X_{i_1}+\cdots+X_{i_k}.
\end{equation}


\begin{lemma}\label{lem-easy} For  any $\bi\in \Sigma_m^*$, it holds that

$\displaystyle (i) \  \log \rho_{\bi}=-\kappa(\bi)\cdot \alpha; \quad  \quad
(ii) \
|\bi|\leq \frac{ |\kappa(\bi) \cdot \balpha |}{\underset{1\leq j \leq m}\min  (X_j\cdot {\balpha}) }.\quad\quad\quad
$
\end{lemma}

\begin{proof} (i) By the definition of $\log_{\blambda}$, we have
$\rho_j=\lambda_1^{X_{j,1}}\cdots \lambda_s^{X_{j,s}}$ for all $j=1,\dots, m$. It follows that
$\log \rho_j=-X_j\cdot \balpha$ for all $j$. Therefore $\log \rho_{\bi}=-\kappa(\bi)\cdot \alpha$.

(ii) This follows from the fact  $ X_j \cdot \balpha>0$ for all $j$.

\end{proof}

\subsection{Cut sets}
For any $t\in (0,1),$    the \emph{cut-set} determined by the threshold $t$ is defined as
\begin{equation}\label{E,t}
{\cal W}_E(t):=\left \{ \mathbf{i}\in \Sigma_m^* :
 \mathbf{\rho}_{\mathbf{i}}\leq t <\mathbf{\rho}_{\mathbf{i}^*} \right \}
\end{equation}
where $\mathbf{i}^*$ is the word obtained by deleting the last letter of $\mathbf{i},$
\textit{i.e.}, $\mathbf{i}^*=i_1,\dots,i_{k-1}$ if $\mathbf{i}=i_1,\dots,i_k.$
(We define $ \rho_{\mathbf{i}^*}=1$ if the length of $\mathbf{i}$ equals $1.$)
 (see \cite{Fal-book2}).

For  $k\in \mathbb{N}^+,$ set
\begin{equation}\label{eq-AkE}
  {\cal A}_{k,E}={\cal A}_k:=\left  \{\log_{\blambda} \rho_{\bi};~ \bi\in    {\cal W}_E(e^{-k})\right\}.
\end{equation}
Then, by \eqref{eq-r2z},  ${\cal A}_{k}$ are subsets of $\CJ:=X_1\N+\cdots +X_m\N$. The following lemma describes
the distribution of ${\cal A}_k$. See Figure \ref{fig-Ak}.

 \begin{lemma}\label{chikaku} (i) For any ${\mathbf b}'\in {\mathcal A}_k$, it holds that
 \begin{equation}\label{equ-b'}
k\leq   {\mathbf b}'\cdot \balpha < k-\log \rho_{\min},
\end{equation}
where   $\rho_{\min}=\underset{1\leq j\leq m}{\min}\rho_j$.

 (ii) There is a constant $C_1$ (independent of $k$) such that
  $$
  d(x, {\cal A}_{k})<C_1
  $$
  for all $x\in {\mathbf C}_X$ with $x \cdot \balpha=k$.
\end{lemma}


\begin{proof} (i)
Take any ${\mathbf b}'=\kappa({\mathbf i}) \in {\mathcal A}_k$. Then
$
e^{-k} \rho_{\min} < \rho_{\mathbf{i}} \leq e^{-k},
$
 so,  taking logarithm  at all sides of the inequality and using Lemma \ref{lem-easy}(i), we obtain \eqref{equ-b'}.

(ii) By Lemma 2.2 in \cite{FRZ14}, there is a constant  $R_0$   such that ${\mathcal J}$ is \emph{$R_0$-relatively dense} in
 ${\mathbf C}_X$ , that is,  for any $x\in {\mathbf C}_X$, there exists $y\in \CJ$
such that $|x-y|<R_0$.

  Take any $x\in  {\mathbf C}_X$ on the hyperplane $x \cdot \balpha =k$.
 Let  ${\mathbf b}$ be a point in ${\mathcal J}$  such that $|x-{\mathbf b}|\leq R_0$, then, by the triangle inequality,
\begin{equation}\label{equ-b}
k-R_0|\balpha|\leq {\mathbf b} \cdot \balpha \leq k+R_0|\balpha|.
\end{equation}
 Let $\bi=i_1 \dots i_l\in \Sigma_m^*$ be a path such that
 $\kappa(\bi)={\mathbf b}.$

If  $\rho_{i_1 \dots i_l}\leq e^{-k}, $ then there exists $p\leq l$ such that
$i_1\dots i_p\in {\cal W}_E(e^{-k})$. Set
${\mathbf b}'=\log_{\blambda}\rho_{i_1\dots i_p}$.
 Then \eqref{equ-b} and   \eqref{equ-b'} imply that
 $$
| ({\mathbf b}-{\mathbf b}')\cdot  \balpha |\leq -\log \rho_{\min}+R_0|\balpha|,
$$
so, since ${\mathbf b}-{\mathbf b}'=X_{i_{p+1}}+\cdots+X_{i_l}$, by Lemma \ref{lem-easy}(ii), we have
$$
|l-p| 
\leq \frac{-\log\rho_{\min}+R_0|\balpha|}{\underset{1\leq j \leq m}\min (X_j\cdot {\balpha})  }:=C',
$$
and
$$|{\mathbf b}-{\mathbf b}'|
\leq C'\underset{1\leq j\leq m}\max|X_j|.
$$

If $\rho_{i_1,\dots,i_l}>e^{-k},$
  then there exists $i_{l+1}\dots i_p\in \Sigma_m^*$ such that
$i_1\dots i_li_{l+1}\dots i_p\in {\cal W}_E(e^{-k}).$
An argument  similar  as above shows that all the above relations still hold.

Hence, we always have
$\displaystyle
|x-{\mathbf b}'|\leq R_0+C'\underset{1\leq j\leq m}\max|X_j|:=C_1,
$
which proves (ii).
\end{proof}

\begin{figure}
  \includegraphics[width=7cm]{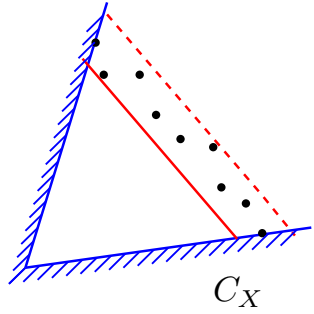}\\
  \caption{ Distribution of ${\mathcal A}_k$. The solid line is the hyperplane $x\cdot \balpha=k$, and
  the dot line is the hyperplane $x\cdot \balpha=k-\log \rho_{\min}$. }
  \label{fig-Ak}
\end{figure}

\subsection{Matchable condition}
Let $E$ and $F$ be two dust-like self-similar sets with contraction vectors
$\boldsymbol{\rho}$ and $\boldsymbol{\tau}$, respectively.
Suppose  $\boldsymbol{\rho}$ and $\boldsymbol{\tau}$ have a
 common pseudo-basis $\boldsymbol{\lambda}=(\lambda_1,\dots,\lambda_s)$.
Let $h$ be a distance on  the group $ \langle \boldsymbol{\rho},\boldsymbol{\tau}\rangle$ defined by
\begin{equation}\label{dis}
h(x_1,x_2):=|\log_{\boldsymbol{\lambda}}x_1-\log_{\boldsymbol{\lambda}}x_2 |.
\end{equation}
Denote $\#A$ the cardinality of a set $A$.

\begin{defi} (\cite{RRW12})
Let $M_0>0$ be a constant and $t\in (0,1).$  We say that ${\cal W}_E(t)$ and ${\cal W}_F(t)$ are $(M_0,h)$-\emph{matchable}, or simply $M_0$-matchable,
if there exists a relation $ {\cal R}\subset {\cal W}_E(t)\times {\cal W}_F(t)$ such that

\begin{itemize}
\item[(i)] $1\leq \# \{ \mathbf{j}:( \mathbf{i}, \mathbf{j}) \in {\cal R}\}\leq M_0$ for any $\mathbf{i}\in {\cal W}_E(t) $,
and\\
 $1\leq \# \{ \mathbf{i}:( \mathbf{i}, \mathbf{j}) \in {\cal R}\}\leq M_0$ for any $\mathbf{j}\in {\cal W}_F(t); $

\item[(ii)] If $ ( \mathbf{i}, \mathbf{j}) \in {\cal R},$ then $h(\rho_{\mathbf{i}},\tau_{\mathbf{j}})\leq M_0.$
\end{itemize}
\end{defi}

The matchable condition is necessary for bi-Lipschitz equivalence.

\begin{thm}\label{thm-match} (\cite{RRW12})
Let $E$ and $F$ be two dust-like self-similar sets.
If $E\sim F,$ then  there exists a constant $M_0$ such that for all
$t\in(0,1),$ ${\cal W}_E(t) $ and $ {\cal W}_F(t)$ are $M_0$-matchable.
\end{thm}


We shall use $\m_E$ to denote the multiplicity function of the higher dimensional Frobenius problem associated with
 $E$, and $\m_F$ the function corresponding to $F$.
Similarly,  let $\gamma_E$ and $\gamma_F$ be the directional growth function corresponding to $E$ and $F$, respectively;
let $\CJ_E$ and $\CJ_F$  be the semi-group corresponding to $E$ and $F$ respectively.

\medskip

\subsection{\textbf{Proof of Theorem \ref{main-1} (ii).}}
Suppose $E\sim F$. Let $M_0$ be the constant in Theorem \ref{thm-match}. Then, for each $k\ge 1,$ there exists an $M_0$-matchable relation ${\cal R}_k$   between $ {\cal W}_E(e^{-k})$ and $ {\cal W}_F(e^{-k}).$

First,  ${\mathbf C}_{X}={\mathbf C}_{Y}$ by Theorem \ref{main-1} (i), and
we denote this common cone by ${\mathbf C}$.

Fix a unit vector $\theta\in {\mathbf C}$ and $k\geq 1$.   Denote
$$
{O}_{k}=\frac{k\theta} {\theta \cdot \balpha },
$$
then $O_k \cdot \balpha=k$.
Let  $w_{k,E}$ be a point in ${\mathcal A}_{k,E}$  such that
\begin{equation}\label{OK-1}
|w_{k,E}-O_{k}|<C_1,
\end{equation}
where $C_1$ is the constant in Lemma \ref{chikaku}; similarly, let $w_{k, F}$
be a point in ${\mathcal A}_{k,F}$ such that
\begin{equation}\label{OK-2}
|w_{k,F}-O_{k}|<C_1.
\end{equation}
(Here we choose  $C_1$   such that
  Lemma \ref{chikaku} holds  for $E$ and $F$ simultaneously.)
 We claim

 \medskip

 \emph{Claim 1. There exists a polynomial $P(x)$ (independent of $\theta$) with positive coefficients such that
 $$
\frac{\bm_E(w_{k,E}) }{\bm_F(w_{k,F})}\leq P(k)
$$
 for all $k\geq 1$ and all unit vector $\theta\in {\mathbf C}$. }

 \medskip
 For a path $\bi=i_1\dots i_\ell \in \Sigma_m^*$, we say $\bi$ \emph{enters} a set $B$
  w.r.t. $E$ if
 $$\log_{\blambda} \rho_{\bi}=X_{i_1}+\cdots+X_{i_\ell}\in B.$$

 Let
$
 I_{k}=\{\mathbf{i};~ \log_{\blambda}\rho_{\bi}=w_{k,E}\}
 $
 be the collection of paths entering the singleton set $\{w_{k,E}\}$. Then, by definition,
   \begin{equation}\label{eq-vs-1}
   {\mathbf m}_E(w_{k,E})=\#I_{k}.
   \end{equation}

   We divide the proof of Claim 1 into three steps.

   \medskip

   \textbf{Step 1. Comparing $\#I_k$ and $\#I_k^*$.}

    Set $I_k^*$ to be the paths in ${\mathcal W}_E(e^{-k})$ which will \emph{eventually enter} $\{w_{k,E}\}$, \emph{i.e.},
$$
I^*_k=\{\bi'\in  {\cal W}_E(e^{-k});~\bi' \text{ is a prefix of some }\bi\in I_k\}.
$$
Write $\bi=\bi'\bi''$. Then, by Lemma \ref{chikaku}(i), since both $\kappa(\bi)$ and $\kappa(\bi')$ belong to ${\mathcal A}_{k,E}$,
$$
\kappa(\bi'')\cdot \balpha=\kappa(\bi)\cdot \balpha-\kappa(\bi')\cdot \balpha\leq -\log \rho_{\min}.
$$
Hence, by Lemma \ref{lem-easy}(ii), there is a constant $K_0$ (independent of $k$ and $\theta$) such that $|\bi|-|\bi'|<K_0$,
and from which we obtain estimates
on the cardinality and  the entering positions of $I_k^*$ as follows.

First,  for any $\bi'\in I_k^*$,  there  are at most $m^{K_0}$ elements in $I_k$  having $\bi'$ as  prefix, so
  \begin{equation}\label{eq-vs-2}
  m^{K_0} (\#I_k^*)\geq \# I_k.
  \end{equation}

Secondly, all elements of $I^*_k$ enter (w.r.t. $E$) the disc with center
$w_{k,E}$ and radius $R_1$, where
$$R_1=K_0\max_{1\leq j\leq m}|X_j|.$$
 In other words,
$\displaystyle |\log_{\blambda}\rho_{\bi'}-w_{k,E}|<R_1 $
for all $\bi'\in I_k^*$.
The above disc $B(w_{k,E}, R_1)$ is the small disc in Figure \ref{discs}.

\medskip

\begin{figure}
  \includegraphics[width=8cm]{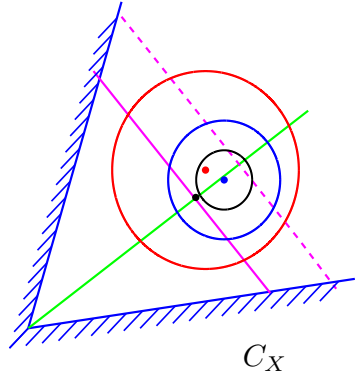}\\
  \caption{The black point is $O_k$, the red point is $w_{k,F}$, and the blue point is $w_{k,E}$;
  the small disc and the medium disc have the common center $w_{k,E}$, their  radii are  $R_1$ and  $R_2$ respectively, the large disc has center $w_{k,F}$ and radius $R_3$.}
  \label{discs}
\end{figure}

  \textbf{Step 2. Comparing $\#I_k^*$ and $\#J_k^*$.}

  Let $J_k^*$ be the set of elements in ${\cal W}_F(e^{-k})$ which has  ${\mathcal R}_k$-relation with   at least one element of $I_k^*$,
\textit{i.e.},
$$J_{k}^*= \{\mathbf{j}\in {\cal W}_F(e^{-k});~
\exists  \mathbf{i} \in I_{k}^* \text { such that }(\mathbf{i},\mathbf{j})\in {\cal R}_k\}.$$

Condition (i)  in the  definition  of the matchable condition implies that
\begin{equation}\label{eq-vs-3}
\# J_{k}^*\geq M_0^{-1}(\#I_{k}^*),
\end{equation}
and Condition (ii) there implies that  for every $\mathbf{j}\in J_{k}^*$,
\begin{equation}\label{IJ-2}
| \log_{\boldsymbol{\lambda} }(\tau_{\mathbf{j}})-w_{k,E}|\leq |\log_{\blambda} (\tau_{\bj})-\log_{\blambda} (\rho_{\bi'})|+|\log_{\blambda}(\rho_{\bi'})-w_{k,E}|\leq M_0+R_1:=R_2,
\end{equation}
where $\mathbf i'$ is any element in $I_k^*$ such that it is ${\mathcal R}_k$-related to ${\mathbf j}$.
We  shall
 call the disc  $B(w_{k,E}, R_2)$ the \emph{medium disc}, see Figure \ref{discs}.
\medskip

\textbf{Step 3. Comparing $\#J_k^*$ and $\m_F(w_{k,F})$.}

Let $G$ be the collection of paths (w.r.t. $F$)  entering the medium disc, precisely,
 $$
 G= \{ \mathbf{j};~| \log_{\boldsymbol{\lambda} }(\tau_{\mathbf{j}})-w_{k, E}|\leq R_2 \}.
$$
 From  \eqref{IJ-2}, we see that $J_k^*\subset G$, so
\begin{equation}\label{eq-vs-4}
\#  G  \geq \# J_{k}^*.
\end{equation}

 Hence we need only compare $\#G$ and $\m_F(w_{k,F})$. This is done by the following lemma.

\begin{lemma} There is a polynomial $\tilde P(x)$ (independent of $\theta$) with positive coefficients such that
$$\#G\leq \tilde P(k) \mathbf{m}_F(w_{k,F}).$$
\end{lemma}

\begin{proof}
We move the center from $w_{k,E}$ to $w_{k,F}$,
and   call the disc with center $w_{k,F}$ and radius $R_3:=R_2+2C_1$ the \emph{large disc}.
It is easy to verify that  the large disc contains  the medium  disc as a subset.

Let $z$ be a point in the large disc. Then
$|z|$ is bounded by
$$
|z|<|O_{k}|+C_1+R_3\leq (\theta\cdot \alpha)^{-1}k+R_3+C_1<ck,
$$
where $c=\max\left \{(\theta\cdot \alpha)^{-1}; ~\theta\in {\mathbf C}_X \text{ and } |\theta|=1\right \}+R_3+C_1.$
By a compact argument, we see that $c<+\infty$. Hence, by Theorem \ref{Q(x)},
  there exists a polynomial $Q(x)$ such that
\begin{equation}\label{eq-wEF-2}
{\mathbf{m}_F(z)}\leq Q(|z|) \cdot {\mathbf{m}_F(w_{k,F}) }
\leq Q(ck) \cdot {\mathbf{m}_F(w_{k,F}) }
\end{equation}
for all $z\in \CJ_F$ in the large disc.

Denoting $N_0$ the number of integer points containing in the large disc,
we have
$$
\begin{array}{rl}
  \#G= &\#\{\text{paths  entering the medium disc w.r.t. $F$}\}\\
\leq & \#\{\text{paths entering the large disc w.r.t. $F$}\}\\
= & \sum_\text{$z\in \CJ_F$ in the large disc} \mathbf{m}_F(z)  \\
\leq &N_0 Q(ck)\mathbf{m}_F(w_{k,F}). \quad \quad (\text{By } \eqref{eq-wEF-2})
\end{array}
 $$
 The lemma is valid by setting $\tilde P(x)=N_0Q(cx)$.
 \end{proof}

Summarizing the above estimates, we obtain
\begin{eqnarray*}
\m_E(w_{k,E})&=& (\#I_k)\leq m^{K_0}(\#I_k^*)\leq m^{K_0}M_0(\# J_k^*)\\
        &\leq & m^{K_0}M_0(\#G)\leq m^{K_0}M_0 \tilde P(k) \m_F(w_{k,F}),
\end{eqnarray*}
which proves Claim 1, where we set $P(x)=m^{K_0}M_0 \tilde P(x)$.

It follows that, for all $k\geq 1$,
\begin{equation}\label{log3}
\frac{\log \bm_E(w_{k,E})}{|O_k|}
\leq \frac{\log (P(k))}{|O_k|}+\frac{\log \bm_F(w_{k,F})}{|O_k|}.
\end{equation}

Let  $z_{k,E}$  be a point in ${\mathcal J}_E$ such that $|O_{k}-z_{k,E}|$
  attains the minimal value and $\m_E(O_k)=\m_E(z_{k,E})$. Since
  $$
  |z_{k, E}-w_{k,E}|\leq |z_{k,E}-O_k|+|O_k-w_{k,E}|<2C_1,
  $$
  again by Theorem \ref{Q(x)}, we have
  $$
  \lim_{k\to \infty} \frac{\log \bm_E(w_{k,E})}{|O_k|}=\lim_{k\to \infty} \frac{\log \bm_E(z_{k,E})}{|O_k|}:=\gamma_{_E}(\theta).
  $$
Similar result holds for $\gamma_{_F}(\theta)$ by the same argument as above. Hence, taking  limits over both sides of (\ref{log3}),  we obtain $ \gamma_{_E}(\theta)\leq  \gamma_{_F}(\theta)$.
Finally, by symmetry, we get the other side inequality and hence  $\gamma_{_E}(\theta)= \gamma_{_F}(\theta)$.
$\Box$

\medskip

 \section{\textbf {Proof of Theorem \ref{main-2} and Example \ref{exam-full} }}\label{sec-final}

\subsection{ Proof of Theorem \ref{main-2}.}  Suppose ${\mathcal D}(\brho)\sim {\mathcal D}(\btau)$. By Theorem \ref{main-1}, we have
 $\gamma_{_X}=\gamma_{_Y}$. Hence, by Theorem \ref{rigidity}, there exist $p$ and $q$ such that
 the $p$-th iteration of $X$ is a permutation of the $q$-th iteration of $Y$.
 It follows that the $p$-th iteration of $\brho$ is a permutation of the $q$-th iteration of $\btau$.

 On the other hand, suppose that $\brho^{p}$ is a permutation of $\btau^{q}$,
 then
 $${\mathcal D}(\brho^{p})\sim {\mathcal D}(\btau^{q}).$$
  Since an IFS and its $n$-th
 iteration define the same invariant set (\cite{Fal-book2}), we have ${\mathcal D}(\brho)\sim {\mathcal D}(\brho^{p})$
 and ${\mathcal D}(\btau)\sim {\mathcal D}(\btau^{q})$. Therefore, ${\mathcal D}(\brho)\sim {\mathcal D}(\btau)$.
$\Box$

\medskip

\subsection{Proof of Example \ref{exam-full}}
Notice  both $\brho$ and $\btau$ are coplanar.
Hence ${\mathcal D}(\brho)\sim {\mathcal D}(\btau)$ implies that  their exist two integers $p,q>0$ such that
the $p$-iteration of $\brho$ is a permutation of the $q$-th iteration of $\btau$.

Take any $j\in \{1,\dots, s\}$. There are $a_j^p$  terms in the vector $\brho^{p}$ being a power of $\lambda_j$, actually, equal to $\lambda_j^{pk_j}$; similarly, there are $b_j^q$ terms in $\btau^{q}$ being a power of $\lambda_j$ and actually equal to $\lambda_j^{q\ell_j}$. We conclude that
$$
a_j^p=b_j^q, \quad \text{ for all } j=1,\dots, s.
$$
Counting the dimensions of $\brho^{p}$ and $\btau^{q}$, we obtain
$$
(a_1+\cdots+a_s)^p=(b_1+\cdots+b_s)^q.
$$
By the convexity of the function
$f(x)=x^{p/q}$ and Jensen's inequality, these equations can hold simultaneously only when $p=q$, and $a_j=b_j$ for all $j$. Hence, we can take $p=q=1$, and it follows that
$\brho$ is a permutation of $\btau$. $\Box$


\begin{thebibliography}{99}
\bibliographystyle{ieee}
\addcontentsline{toc}{chapter}{Bibliography}

\bibitem{Alf05} J. Ramirez Alfonsin,  \emph{The Diophantine Frobenius problem,} Oxford Univ. Press, 2005.

\bibitem{CP} D.~Cooper and T.~Pignataro, {\it On the shape of Cantor
sets}, J. Differential Geom., \textbf{28} (1988), 203--221.

\bibitem{DS}
G.~David and S.~Semmes, {\it Fractured fractals and broken dreams :
self-similar geometry through metric and measure}, Oxford Univ.\
Press,~1997.

\bibitem{DengGT1} G. T. Deng and X. G. He, \emph{Lipschitz equivalence of fractal sets in $\R$,
} Sci. China. Math., \textbf{55} (2012), 2095--2107.

\bibitem{Deng11} J. Deng, Z. Y. Wen, Y. Xiong and  L. F. Xi,
\emph{Bilipschitz embedding of self-similar sets,}
J. Anal. Math., \textbf{114} (2011), 63--97.

\bibitem{Fal-book2}  K. J.~Falconer, {\it Fractal Geometry: Mathematical
Foundations and Applications},   John Wiley \& Sons, New York,
 1990.


\bibitem{FaMa89} K. J.~Falconer and D. T.~Marsh, {\it Classification of
quasi-circles by Hausdorff dimension}, Nonlinearity, \textbf{2}
(1989), 489--493.


\bibitem{FaMa92} K. J.~Falconer and D. T.~Marsh, {\it On the Lipschitz equivalence of
Cantor sets}, Mathematika, \textbf{39} (1992), 223--233.

\bibitem{FRZ14} A. H. Fan, H. Rao and Y. Zhang,
\emph{Higher dimensional Frobenius problem: maximal saturated cone,
growth function and rigidity,} Preprint 2014.

\bibitem{Hun80} T. W.~Hungerford, {\it Algebra}, Graduate Texts in Mathematics
{\bf 73}, Springer, New York, 1980.

\bibitem{Hut81} J. E. Hutchinson,  {\it Fractals and self-similarity},
Indiana Univ. Math. J.,  {\bf 30}  (1981),  713--747.

\bibitem{Lau13} J. J. Luo and K. S. Lau,
{\it Lipschitz equivalence of self-similar sets and hyperbolic boundaries},
Adv. Math., \textbf{235} (2013), 555--579.

\bibitem{Llo09} M. Llorente and P. Mattila,
\emph{Lipschitz equivalence of subsets of self-conformal sets,}
Nonlinearity, \textbf{23} (2010), 875--882.

\bibitem{LLM} B. M. Li, W. X. Li and J. J.  Miao,
\emph{Lipschtiz  Equivalence of Mcmullen sets,}
Fractals, \textbf{21} (2013), 3--11.

\bibitem{MS08}
P. Mattila and P. Saaranen, {\it Ahlfors-David regular sets and
bilipschitz maps}, Ann. Acad. Sci. Fenn. Math., \textbf{34} (2009), 487--502.



\bibitem{RRW12}
H. Rao, H. J. Ruan, and Y. Wang,  {\it Lipschitz equivalence of Cantor sets and algebraic properties of contraction ratios},  Trans. Amer. Math. Soc.,  \textbf{364} (2012), 1109-1126.

\bibitem{RRW13} H. Rao, H. J. Ruan and Y. Wang, \emph{Lipschitz equivalence of self-similar sets: algebraic and geometric properites,}
Contemp. Math., \textbf{600} (2013). 


\bibitem{RRX06} H.~Rao, H. J.~Ruan and L. F.~Xi, {\it Lipschitz equivalence of self-similar
sets}, C. R. Acad. Math. Sci. Paris., \textbf{342} (2006),
191--196.

\bibitem{RRY08} H.~Rao, H. J.~Ruan and Y. M.~Yang, {\it Gap sequence, Lipschitz
equivalence and box dimension of fractal sets}, Nonlinearity,
\textbf{21} (2008), 1339--1347.


\bibitem{Ro10}  K. A. Roinestad, {\it Geometry of fractal squares},  Ph.D. Thesis, The Virginia Polytechnic Institute and State University (2010).

\bibitem{RWX12} H. J. Ruan, Y. Wang, L. F. Xi,
{\it Lipschitz equivalence of self-similar sets with touching structures},
 Nonlinearity,  \textbf{27} (2014), 1299--1321.


\bibitem{DengGT2} Z. X. Wen, Z. Y. Zhu and G. T. Deng, {\it Lipschitz equivalence of a class of general Sierpinski carpets,}
J. Math. Anal. Appl., \textbf{385} (2012), 16--23.



\bibitem{Xi09} L. F.~Xi, {\it Lipschitz equivalence of dust-like self-similar
sets},  Math. Z., \textbf{266} (2010), 683--691.

\bibitem{Xi04} L. F.~Xi, {\it Lipschitz equivalence of self-conformal sets}, J.
London Math. Soc.(2), \textbf{70} (2004), 369--382.


\bibitem{XiRu07} L. F.~Xi and H. J.~Ruan,  {\it Lipschitz equivalence of generalized
$\{1,3,5\}-\{1,4,5\}$ self-similar sets}, Sci. China Ser. A,
\textbf{50} (2007), 1537--1551.

\bibitem{XiRu08} L. F.~Xi and H. J.~Ruan, {\it Lipschitz equivalence of self-similar sets satisfying the strong separation
condition} (in Chinese), Acta Math. Sinica (Chin. Ser.), \textbf{51}
(2008), 493--500.


\bibitem{XiXi10} L. F. Xi and Y. Xiong, {\it Self-similar sets with initial cubic patterns},
 C. R. Math. Acad. Sci. Paris, \textbf{348} (2010),  15-20.

\bibitem{XiXi12}L.  F. Xi and  Y. Xiong, {\it Lipschitz equivalence of fractals generated by nested cubes},
 Math Z.,  \textbf{271} (2012), 1287--1308.

\bibitem{XiXi13} L. F. Xi and Y. Xiong,  \emph{Lipschitz Equivalence Class, Ideal Class and the Gauss Class Number Problem,}
Preprint 2013 (arXiv:1304.0103 [math.MG]).



\end{thebibliography}
\end{document}